\documentclass[12pt]{article}

\usepackage{amssymb}
\usepackage{amsmath,amsthm,amsfonts}
\usepackage{graphicx}
\usepackage{ mathrsfs}
\newtheorem{theorem}{Theorem}
\usepackage{hyperref}
\usepackage{subfigure}
\newtheorem{remark}{Remark}
\newtheorem{example}{Example}

\begin{document}

    \title{ON FRACTIONAL SPHERICALLY RESTRICTED HYPERBOLIC
        DIFFUSION RANDOM FIELD}
    \author{N.Leonenko \\
        Cardiff University\\
        Senghennydd Road\\
        CF24 4AG Cardiff, UK\\
        LeonenkoN@Cardiff.ac.uk
        \and A.Olenko\\
        La Trobe University\\
        Melbourne, Australia\\
        a.olenko@latrobe.edu.au
        \and J.Vaz \\
        University of Campinas\\
        IMECC\\
        SP 13083-859 Campinas, Brazil\\
        vaz@unicamp.br}
    \maketitle

    \begin{abstract}
        The paper investigates solutions of the fractional hyperbolic diffusion equation in its most general form with two fractional derivatives of distinct orders. The solutions are given as spatial-temporal homogeneous and isotropic random fields and  their  spherical restrictions are studied.  The spectral representations of these fields are derived and the associated angular spectrum is analysed.  The obtained mathematical results are illustrated by numerical examples. In addition, the numerical investigations assess the dependence of the covariance structure and other properties of these fields on the orders of fractional derivatives.
        \vspace{2mm}
        \newline
        AMS classification: 60G60, 60G15, 60D05, 60K05
        \vspace{2mm}
        \newline
        Keywords: Spherical random fields, Fractional hyperbolic diffusion equation, Caputo derivative, Fractional
        telegraph equation, Angular spectrum, Spectral theory, Spatio-temporal data
    \end{abstract}

    \bigskip\ \ \addtolength{\baselineskip}{1em}

    \section{\protect\bigskip Introduction}

    Spherical random fields have been used for modelling various phenomena in
    areas such as earth sciences, for example, in geophysics and
    climatology \cite{CG, CG1, Fisher, M19, Oh}, and cosmology,  see
    \cite{BKLO, BKLO1, Cabella, MP} and the references therein. In fact, the application of statistical methods in
    cosmology \cite{Cabella} has become increasingly important due to the many
    experimental data obtained in recent years~\cite{Adam}. Spherical random
    fields are of particular interest for modelling and analysis of Cosmic Microwave
    Background (CMB) radiation \cite{BKLO1, BNO, LNO, MP}.  The CMB is a spatially isotropic radiation field spread throughout the
    universe, that was originated around 14 billion years ago \cite{Dodelson,Weinberg}. It is the main
    source of information we have about  initial phases of the universe. The CMB
    radiation can be mathematically modelled as an isotropic, mean-square
    continuous spherical random field, which has a spectral
    representation by means of spherical harmonics. Consequently,
    models of temporal spherical random fields, in addition to their
    innate theoretical interest, have some practical
    applications in the studies of the CMB radiation evolution.

    Such models were recently provided in \cite{BKLO, BKLO1,LNO}, where stochastic
    hyperbolic diffusion equations were used to describe changes of the spherical random fields over time.
    The hyperbolic heat equation is formally identical to the linear telegraph
    equation presented by Heaviside in his study of transmission lines. It
    was introduced by Cattaneo \cite{Cattaneo}  to impose a bounded
    speed of propagation for the temperature disturbances, in contrast
    with the classical parabolic heat equation, that has an unbounded
    propagation speed.
    While the classical diffusion equation results from the
    conversation law $u_{t}+{\rm{div}}q=0$ and Fick's law $q+k\,{\rm{grad}}u=0$,
    Cattaneo equation results from the modification of Fick's law by introducing
    a term proportional to the first order derivative of the flux $q$, that is, $%
    \tau q_{t}+q+k\,{\rm{grad}}u=0$.

    The boundedness of the speed propagation is desirable
    because the large-scale coherent structures that are observed in the CMB are
    believed to be the remains of the waves in the plasma universe. In \cite%
    {BKLO} the explicit solution of the model was given in terms of series of
    elementary functions, and therefore it could be useful for various
    qualitative and numerical studies. For more details and references on the
    telegraph equations, or hyperbolic diffusion equation, consult \cite{BKLO,
        BKLO1, K, KR, Pov}, among others.

    On the other hand, it is  known that deviations from the standard diffusive behavior occur in many situations \cite{Bouchaud}. Among the various models of anomalous behavior, an important approach is based on fractional differential equations \cite{CapelasCNSNS, Negrete,  DLO}.
    The calculus of non-integer
    order called fractional calculus, attracted an increasing interest over the
    last decades, specially for the modelling of phenomena involving memory effects, see, for example, anomalous transport
    \cite{Negrete}, problems with dissipation \cite{Mainardi2010}, \cite{Bho} etc. However,
    these models are not unique in the sense that there are numerous non-equivalent definitions of a fractional derivative in the literature.

    It is natural at this point to think of Cattaneo's approach to
    the hyperbolic diffusion equation within the framework of the continuous-time random walk (CTRW). The key point in Cattaneo's approach was to modify the constitute
    equation (Fick's law) by introducing a term proportional to the first order
    derivative of the flux. In \cite{Compte} Compte and Metzler discussed
    possibilities for the generalization of the Cattaneo equation, and one of
    these possibilities is to consider the CTRW scenario of fractal time random
    walk, which gives an equation which can be written in terms of Caputo
    fractional derivatives in the time variable.

    If one considers another modification of Cattaneo
    equation, replacing the ordinary derivative by fractional derivatives of
    the Caputo type, see (\ref{2.1}) below, then
    \begin{equation*}
        \frac{\partial ^{\alpha }u}{\partial t^{\alpha }}+{\rm{div}}q=0,\qquad \tau
        \frac{\partial ^{\beta }q}{\partial t^{\beta }}+q+k\,{\rm{grad}}u=0,
    \end{equation*}%
    and the resulting modified Cattaneo equation is
    \begin{equation}
        \tau \frac{\partial ^{\beta }\;}{\partial t^{\beta }}\frac{\partial ^{\alpha
            }u}{\partial t^{\alpha }}+\frac{\partial ^{\alpha }u}{\partial t^{\alpha }}%
        -k\nabla ^{2}u=0.  \label{eq.0}
    \end{equation}

    It was shown \cite{Beghin} that for $\alpha +\beta >1$
    \begin{equation*}
        \frac{\partial ^{\beta }\;}{\partial t^{\beta }}\frac{\partial ^{\alpha }u}{%
            \partial t^{\alpha }}=\frac{\partial ^{\alpha }\;}{\partial t^{\alpha }}%
        \frac{\partial ^{\beta }u}{\partial t^{\beta }}=\frac{\partial ^{\alpha
                +\beta }u}{\partial t^{\alpha +\beta }}
    \end{equation*}%
    if $u(t)$ is analytic and such that, for any $n\in \mathbb{N}$, $%
    |u^{(n)}(0)|<K^{n}$, for some constant $K>0$, and $u^{(1)}(0)=0.$ Under these conditions, (\ref{eq.0}) becomes
    \begin{equation*}
        \tau \frac{\partial ^{\alpha +\beta }u}{\partial t^{\alpha +\beta }}+\frac{%
            \partial ^{\alpha }u}{\partial t^{\alpha }}-k\nabla ^{2}u=0.
    \end{equation*}

    The present paper is a continuation of the line of research of the work \cite%
    {BKLO},\cite{BKLO1} and \cite{LV}. Our main objective is to study the
    fundamental solutions to fractional hyperbolic diffusion equation in the
    time variable using the Caputo derivative, and its properties. The exact
    solutions of  fractional hyperbolic diffusion equations with random initial conditions
    are derived. Then, the angular spectrum of the solution to the spherical fractional hyperbolic
    diffusion equations restricted to sphere from homogeneous and isotropic
    random field is obtained. It is given as a solution of fractional hyperbolic diffusion
    equation with random random initial conditions under very general assumptions about indexes $%
    \alpha $ and $\beta $ of fractional equation, namely, $0<\alpha \leq 1$ and $
    1<\alpha +\beta \leq 2.$

    Time-fractional telegraph equations with Caputo fractional derivatives and $\beta =\alpha \in (0,1]$ were
    considered  in \cite{OB} and \cite{DOT},
    while with Hilfer \ and Hadamar fractional derivatives in~\cite{SGO}. They
    provided the Fourier transform of some Cauchy problems for these equations
    as well as probabilistic interpretations in some specific cases. In the framework of spherical random fields
    satisfying random initial conditions, the fundamental solution of the
    hyperbolic diffusion equation were derived in~\cite{BKLO} and \cite%
    {BKLO1}, while the fractional version was investigated in~\cite{LV} \
    (again only for the case of $\alpha \in (0,1],~\alpha =\beta ).$
    Obviously, for Caputo fractional derivatives the results of this paper are
    more general since it is only assumed that $0<\alpha \leq 1$ and $1<\alpha +\beta \leq 2.$

    This paper is organized as follows. Section~\ref{sec2} introduces the required notations and the initial-value problem for fractional hyperbolic diffusion equations. Then,  it provides various  results about the Fourier transforms of the solutions to the equations. The solutions are given using these Fourier transforms as a stochastic integrals.  Section~\ref{sec3} investigates restrictions of the solutions to the unit sphere. Specifications of the main results and important particular cases are considered in examples. Section~\ref{sec4} presents simulation studies that illustrate properties of the solutions with respect to orders of fractional derivatives.

    All numerical computations and simulations in this paper were performed using the software R version 4.3.1 and Python version 3.11.5. The HEALPix representation for spherical data was used, see \mbox{\url{http://healpix.sourceforge.net}}.  The Python package "healpy" was utilized for the computing Laplace series coefficients to create spherical maps. The R package "rcosmo", see \cite{fryer2019rcosmo} and \cite{fryer2018rcosmo}, was employed to visualise the obtained spherical fields. The R and Python code used for numerical examples in Section~\ref{sec4} are freely available in the folder "Research materials" from the  website \mbox{\url{https://sites.google.com/site/olenkoandriy/}}.

\section{Fractional Hyperbolic Diffusion Equation}\label{sec2}

The Caputo fractional derivative of order $\gamma ,$ with $(n-1)<\gamma <n,\
n\in \mathbb{N},$ is defined as
\begin{equation}
\mathrm{D}_{t}^{\gamma }q(x,t)=\frac{\partial ^{\gamma }}{\partial t^{\gamma
}}q(x,t)=\frac{1}{\Gamma (n-\gamma )}\int_{0}^{t}\frac{q^{(n)}(x,\tau )}{%
(t-\tau )^{\gamma +1-n}}d\tau ,  \label{2.1}
\end{equation}%
where \ $q^{(n)}$ denotes the partial derivative of order $n$ of $%
q(x,t),x\in
\mathbb{R}
^{3},\ t>0,$ and \ $\Gamma (\cdot)$ is the gamma function. Various properties of
Caputo derivatives can be found in \cite{GLM}, \cite{MS1} and \cite{P}.

We consider the hyperbolic diffusion equation%
\begin{equation}
\frac{1}{c^{2}}\frac{\partial ^{\alpha +\beta }}{\partial t^{\alpha +\beta }}%
q(x,t)+\frac{1}{D}\frac{\partial ^{\alpha }}{\partial t^{\alpha }}%
q(x,t)=k^{2}\Delta q(x,t),  \label{2.2}
\end{equation}%
where  $\Delta $ is the Laplacian in $\mathbb{R}^{3},$ $0<\alpha \leq 1,$ $1<\alpha +\beta \leq 2,$ $D>0,$ and $c>0.$ The random field $q(x,t)=q(x,t,\omega
),$ $\omega \in \mathrm{\Omega },$ satisfies the random
initial conditions:
\begin{equation}
\left. q(x,t)\right\vert_{t=0}=\eta (x),\quad \left. \frac{\partial }{\partial t}q(x,t)\right\vert
_{t=0}=0,  \label{2.3}
\end{equation}%
where  the random field $\eta (x)=\eta (x,\omega ),$ $x\in
\mathbb{R}^{3},\ \omega \in \mathrm{\Omega },$ defined on a suitable complete
probability space $(\mathrm{\Omega },\mathcal{F},\mathrm{P}),$ is assumed to be measurable, mean-square continuous, wide-sense
homogeneous and isotropic with zero mean and the covariance function $\mathrm{Cov%
}(\eta (x),\eta (y))=B(\left\Vert x-y\right\Vert ),\ x\in
\mathbb{R}
^{3},\ y\in
\mathbb{R}
^{3}.$

The covariance function $B(\cdot)$ has the following representation, see \cite[p.12]{IL} and \cite[p.18]{L},
\begin{equation}
B(\left\Vert x-y\right\Vert )=\int_{
\mathbb{R}
^{3}}\cos (\left\langle \lambda ,x-y\right\rangle )F(d\lambda
)=\int_{0}^{\infty }\frac{\sin (\mu \left\Vert x-y\right\Vert )}{\mu
\left\Vert x-y\right\Vert }G(d\mu ),  \label{2.4}
\end{equation}%
for some bounded, non-negative measures $F(\cdot)$ on the measurable space $(\mathbb{R}^{3},\mathcal{B}(
\mathbb{R}
^{3}))$ and $G(\cdot)$ on $\left( \mathbb{R}
_{+}^{1},\mathcal{B}(
\mathbb{R}
_{+}^{1})\right),$ such that%
\begin{equation*}
F(%
\mathbb{R}
^{3})=G\left( [0,\infty )\right) =B(0),\quad G(\mu )=\int_{\left\Vert
\lambda \right\Vert <\mu }F(d\lambda ).
\end{equation*}

Then, there exists a complex-valued orthogonal random measure $%
Z(\cdot)=Z(\cdot,\omega ),$ $\omega \in \mathrm{\Omega },$ such that for every $x\in
\mathbb{R}
^{3}$, the random field $\eta \left( x\right) $ has the spectral
representation%
\begin{equation}
\eta (x)=\int_{
\mathbb{R}
^{3}}e^{\mathrm{i}\left\langle \lambda ,x\right\rangle }Z(d\lambda ),\
\mathrm{E}\left\vert Z(\Delta )\right\vert ^{2}=F(\Delta ),\ \Delta \in
\mathcal{B}(\mathbb{R}^{3}).  \label{2.5}
\end{equation}

Let $Y_{lm}(\theta ,\varphi ),\ \theta \in \lbrack 0,\pi ],\ \varphi \in
\lbrack 0,2\pi ],\ l=0,1,2,\ldots ,m=-l,\ldots ,l,$ be complex spherical
harmonics defined by the relation%
\begin{equation*}
Y_{lm}(\theta ,\varphi )=(-1)^{m}\left[ \frac{(2l+1)(l-m)!}{4\pi (l+m)!}%
\right] ^{1/2}e^{\mathrm{i}m\varphi }P_{l}^{m}(\cos \theta ),
\end{equation*}%
where $P_{l}^{m}(\cdot)$ are the associated Legendre polynomials.

The Bessel function $J_{\nu}(\cdot)$ of the first $k$ and of order $\,\nu $
is defined by%
\begin{equation*}
J_{\nu }(\mu )=\sum_{n=0}^{\infty }\frac{(-1)^{n}}{n!\Gamma (n+\nu +1)}%
\left( \frac{\mu }{2}\right) ^{2n+\nu }.
\end{equation*}

By the addition theorem for Bessel functions (\cite[p.14]{IL} and \cite[p.20]{L}) and the Karhunen theorem (\cite[p.10]{L}):%
\begin{equation}
\eta (x)=\pi \sqrt{2}\sum_{l=0}^{\infty }\sum_{m=-l}^{l}Y_{lm}(\theta
,\varphi )\int_{0}^{\infty }\frac{J_{l+\frac{1}{2}}(\mu r)}{(\mu r)^{1/2}}%
Z_{lm}(d\mu ),  \label{2.6}
\end{equation}%
where $Z_{lm}(\cdot)=Z_{lm}(\cdot,\omega ),$ $\omega \in \Omega $, is a family of
complex-valued random measures on $\left(
\mathbb{R}
_{+}^{1},\mathcal{B}(
\mathbb{R}
_{+}^{1})\right) ,$ such that
\begin{equation*}
\mathrm{E}[Z_{lm}(\Delta _{1})\overline{Z_{l^{\prime }m^{\prime }}(\Delta
_{2})}]=\delta _{ll^{\prime }}\delta _{mm^{\prime }}G(\Delta _{1}\cap \Delta
_{2}),\ \Delta _{i}\in \mathcal{B}(
\mathbb{R}_{+}^{1}),\ i=1,2.
\end{equation*}

The stochastic integrals in (\ref{2.5}) and (\ref{2.6}) are interpreted as $%
\mathscr{L}_{2}(\mathrm{\Omega })$ integrals with structural measures $F$ and
$G$ respectively.

Let us consider the non-random initial conditions of the form%
\begin{equation}
q(x,t)|_{t=0}=\delta (x),\qquad \left. \frac{\partial }{\partial t}%
q(x,t)\right\vert _{t=0}=0,  \label{2.7}
\end{equation}%
where $\delta (x)$ is the Dirac delta function.

Let $Q(x,t),x\in
\mathbb{R}
^{3},t>0,$ be the fundamental solution (or the Green's function) of the
initial-value problem (\ref{2.2}) and (\ref{2.7}), and let%
\begin{equation}
H(\kappa,t)=\int_{
\mathbb{R}
^{3}}e^{i\left\langle \kappa ,x\right\rangle }Q(x,t)dx,\ \kappa \in\mathbb{R}
^{3},\ t\geq 0,  \label{2.8}
\end{equation}%
be its Fourier transform.

The tree-parametric Mittag-Leffler function or Prabhakar function, see \cite%
[ p. 97]{GLM} and~\cite{V1}, is defined by%
\begin{equation}
{\rm{E}}_{a,b}^{\zeta}(z)=\sum_{k=0}^{\infty }\frac{(\zeta)_{k}\cdot z^{k}}{%
k!\,\Gamma (ak+b)},  \label{2.9}
\end{equation}%
where $\zeta>0,\ {\rm{Re}}(a)>0,$ ${\rm{Re}}(b)>0$ and $(\zeta)_{k}=\Gamma
(\zeta+k)/\Gamma (\zeta)$ is the Pochhammer symbol.

The following theorem derives the Fourier transform (\ref{2.8}) in terms of
the function defined by (\ref{2.9}). Later on, this result will be used to obtain the solution $q(x,t,\omega ),\
x\in
\mathbb{R}
^{3},$ $\omega \in \mathrm{\Omega },$ $t\geq 0,$ of the initial-value problem (\ref%
{2.2}) with the initial conditions given in (\ref{2.3})

\begin{theorem}\label{th1} The Fourier transform {\rm(\ref{2.8})} of the initial-value
problem {\rm(\ref{2.2})} and {\rm(\ref{2.7})} is given by the formula%
\begin{eqnarray*}
H(\kappa,t) &=&1+\sum_{n=0}^{\infty }\left( -\left\Vert \kappa\right\Vert
^{2}c^{2}t^{\alpha +\beta }\right) ^{n+1}{\rm{E}}_{\beta ,(\alpha +\beta
)(n+1)+1}^{n+1}\left( -\frac{c^{2}}{D}t^{\beta }\right)  \notag \\
&=&1-\left\Vert \kappa\right\Vert ^{2}c^{2}t^{\alpha +\beta }\sum_{m=0}^{\infty
}\sum_{n=0}^{\infty }\binom{m}{n}\frac{\left( -\frac{c^{2}}{D}t^{\beta
}\right) ^{m}\left( \left\Vert \kappa\right\Vert ^{2}Dt^{\alpha }\right) ^{n}}{%
\Gamma (\beta m+\alpha n+\alpha +\beta +1)},  \notag
\end{eqnarray*}%
where $1<\alpha +\beta \leq 2,\ 0<\alpha \leq 1,\ x\in\mathbb{R}
^{3}, \ t\geq 0.$
\end{theorem}
\begin{proof} The Fourier transformed version of (\ref{2.2}) is
\begin{equation*}
\frac{1}{c^{2}}\frac{\partial ^{\alpha +\beta }}{\partial t^{\alpha +\beta }}%
H(\kappa,t)+\frac{1}{D}\frac{\partial ^{\alpha }}{\partial t^{\alpha }}%
H(\kappa,t)=-\left\Vert \kappa\right\Vert ^{2}H(\kappa,t)  \label{2.11}
\end{equation*}%
with
\begin{equation}
H(\kappa,0)=1,\quad \frac{\partial H(\kappa,0)}{\partial t}=0.  \label{2.12}
\end{equation}

Let $\tilde{H}(s)=\mathcal{L} \lbrack H(\kappa,t);s]$ denotes the Laplace
transform of $H(\kappa,t).$ Then, using \cite[(4.8)]{DOT} one obtains
    \begin{equation*}
\tilde{H}(s)=\frac{s^{\alpha +\beta -1}+c^{2}D^{-1}s^{\alpha -1}}{s^{\alpha
+\beta }+c^{2}D^{-1}s^{\alpha }+\left\Vert \kappa\right\Vert ^{2}c^{2}}=\frac{1}{s%
}-\frac{\left\Vert \kappa\right\Vert ^{2}c^{2}}{s(s^{\alpha +\beta
}+c^{2}D^{-1}s^{\alpha }+\left\Vert \kappa\right\Vert ^{2}c^{2})}.
\end{equation*}

For $s$ such that
\begin{equation}\label{cond1}
\left\vert \frac{\left\Vert \kappa\right\Vert ^{2}c^{2}}{s^{\alpha +\beta
}+c^{2}D^{-1}s^{\alpha }}\right\vert <1
\end{equation}
it can be written as
\begin{equation*}
\tilde{H}(s)=\frac{1}{s}+\sum_{n=0}^{\infty }\left( -\left\Vert
\kappa\right\Vert ^{2}c^{2}\right) ^{n+1}\frac{s^{-\alpha (n+1)-1}}{(s^{\beta
}+c^{2}D^{-1})^{n+1}}.
\end{equation*}

It is known \cite[(5.1.33)]{GLM} that
\begin{equation*}
\mathcal{L} \left[ t^{b-1}{\rm{E}}_{a,b}^{\zeta}(-\mu t^{a });s\right] =%
\frac{s^{a\zeta-b}}{\left( s^{\alpha }+\mu \right) ^{\zeta}},
\end{equation*}%
where the function ${\rm{E}}_{a,b}^{\zeta}$ is defined by (\ref{2.9}).

Using the substitution
\begin{equation*}
a=\beta ,\quad b=(\alpha +\beta )(n+1)+1,\quad \zeta=n+1,
\end{equation*}%
one can write $H(\kappa,t)=\mathcal{L}^{-1}[\tilde{H}(s);t]$ as
\begin{equation*}
H(\kappa,t)=1+\sum_{n=0}^{\infty }(-\left\Vert \kappa\right\Vert ^{2}c^{2}t^{\alpha
+\beta })^{n+1}{\rm{E}}_{\beta ,(\alpha +\beta
)(n+1)+1}^{n+1}(-c^{2}D^{-1}t^{\beta }).
\end{equation*}
It follows by Lerch's theorem that $H(\kappa,t)$ has the Laplace transform $\tilde{H}(s)$ not only for the range given by (\ref{cond1}), but also for all values of $s.$

The above expression for $H(\kappa,t)$ can be rewritten in a form that is  convenient for numeric  computations. Using (\ref{2.9}) one obtains
\begin{equation*}
H(\kappa,t)=1+\sum_{n=0}^{\infty }\sum_{k=0}^{\infty }\frac{(n+k)!\,(\left\Vert
\kappa\right\Vert ^{2}Dt^{\alpha })^{n+1}(-c^{2}D^{-1}t^{\beta })^{n+k+1}}{%
n!\,k!\,\Gamma \lbrack \beta (n+k+1)+\alpha (n+1)+1]},
\end{equation*}%
and using $m=n+k$ as index of summation,
\begin{equation*}
H(\kappa,t)=1-\left\Vert \kappa\right\Vert ^{2}c^{2}t^{\alpha +\beta
}\sum_{m=0}^{\infty }\sum_{n=0}^{m}\binom{m}{n}\frac{(-c^{2}D^{-1}t^{\beta
})^{m}(\left\Vert \kappa\right\Vert ^{2}Dt^{\alpha })^{n}}{\Gamma (\beta m+\alpha
n+\alpha +\beta +1)},
\end{equation*}
which completes the proof.\end{proof}
The following result gives another representation of the Fourier transform via a contour integral.
\begin{theorem}\label{th2} The Fourier transform {\rm(\ref{2.8})} of the initial-value
    problem {\rm(\ref{2.2})} and {\rm(\ref{2.7})} is given by the formula%
\begin{equation}
    H(\kappa,t)=\frac{1}{2\pi \mathrm{i}}\int_{\rm{Ha}}\text{e}^{\xi }\xi ^{\alpha
        -1}\frac{\xi ^{\beta }+c^{2}D^{-1}t^{\beta }}{\xi ^{\alpha +\beta
        }+c^{2}D^{-1}\xi ^{\alpha }t^{\beta }+\left\Vert \kappa\right\Vert
        ^{2}c^{2}t^{\alpha +\beta }}\,d\xi ,  \label{2.18}
\end{equation}
where {\rm$\text{Ha}$} is the Hankel contour.
\end{theorem}
\begin{proof}
It follows from the contour integral representation of the reciprocal gamma
function
\begin{equation*}
\frac{1}{\Gamma (z)}=\frac{1}{2\pi \mathrm{i}}\int_{\text{Ha}}\text{e}^{\xi
}\xi ^{-z}\,d\xi ,
\end{equation*}%
that
\begin{equation*}
H(\kappa,t)=\left[ 1-\frac{\left\Vert \kappa\right\Vert ^{2}c^{2}t^{\alpha +\beta }}{%
2\pi \mathrm{i}}\int_{\text{Ha}}\text{e}^{\xi }\xi ^{-(\alpha +\beta
+1)}\sum_{m=0}^{\infty }\sum_{n=0}^{m}\binom{m}{n}(-c^{2}D^{-1}t^{\beta }\xi
^{-\beta })^{m}(\left\Vert \kappa\right\Vert ^{2}Dt^{\alpha }\xi ^{-\alpha
})^{n}\,d\xi \right].
\end{equation*}%
As the internal sum can be simplified as
\begin{equation*}
\sum_{n=0}^{m}\binom{m}{n}(\left\Vert \kappa\right\Vert ^{2}Dt^{\alpha }\xi
^{-\alpha })^{n}=(1+\left\Vert \kappa\right\Vert ^{2}Dt^{\alpha }\xi ^{-\alpha
})^{m},
\end{equation*}%
then
\begin{equation*}
H(\kappa,t)=1-\frac{\left\Vert \kappa\right\Vert ^{2}c^{2}t^{\alpha +\beta }}{2\pi
\mathrm{i}}\int_{\text{Ha}}\text{e}^{\xi }\xi ^{-(\alpha +\beta
+1)}\sum_{m=0}^{\infty }[(-c^{2}D^{-1}t^{\beta }\xi ^{-\beta })(1+\left\Vert
\kappa\right\Vert ^{2}Dt^{\alpha }\xi ^{-\alpha })]^{m}\,d\xi .
\end{equation*}%
It is convenient now to change the integration variable from $\xi $ to $\rho
=\xi /t$ (for $t>0$), that is,
\begin{equation*}
H(\kappa,t)=\left[ 1-\frac{\left\Vert \kappa\right\Vert ^{2}c^{2}}{2\pi \mathrm{i}}%
\int_{\text{Ha}}\text{e}^{\rho t}\rho ^{-(\alpha +\beta
+1)}\sum_{m=0}^{\infty }\left[ -c^{2}D^{-1}\frac{(\rho ^{\alpha }+\left\Vert
\kappa\right\Vert ^{2}D)}{\rho ^{\alpha +\beta }}\right] ^{m}\,d\rho \right] .
\end{equation*}%
If the contour is taken in a such a way that
\begin{equation*}
\left\vert c^{2}D^{-1}\frac{(\rho ^{\alpha }+\left\Vert \kappa\right\Vert ^{2}D)}{%
\rho ^{\alpha +\beta }}\right\vert <1,
\end{equation*}%
then it follows that
\begin{equation*}
H(\kappa,t)=1-\frac{\left\Vert \kappa\right\Vert ^{2}c^{2}}{2\pi \mathrm{i}}\int_{%
\text{Ha}}\frac{\text{e}^{\rho t}\rho ^{-1}}{\rho ^{\alpha +\beta
}+c^{2}D^{-1}\rho ^{\alpha }+\left\Vert \kappa\right\Vert ^{2}c^{2}}\,d\rho .
\end{equation*}%
Let us return to the $\xi $ variable of the integration, that is,
\begin{equation*}
H(\kappa,t)=1-\frac{\left\Vert \kappa\right\Vert ^{2}c^{2}t^{\alpha +\beta }}{2\pi
\mathrm{i}}\int_{\text{Ha}}\frac{\text{e}^{\xi }\xi ^{-1}}{\xi ^{\alpha
+\beta }+c^{2}D^{-1}\xi ^{\alpha }t^{\beta }+\left\Vert \kappa\right\Vert
^{2}c^{2}t^{\alpha +\beta }}\,d\xi .
\end{equation*}%
By the contour representation of the gamma function
\begin{equation*}
1=\frac{1}{\Gamma (1)}=\frac{1}{2\pi \mathrm{i}}\int_{\text{Ha}}\text{e}%
^{\xi }\xi ^{-1}\,d\xi
\end{equation*}%
and therefore
\begin{equation*}
H(\kappa,t)=\frac{1}{2\pi \mathrm{i}}\int_{\text{Ha}}\frac{\text{e}^{\xi }}{\xi }%
\left[ 1-\frac{\left\Vert \kappa\right\Vert ^{2}c^{2}t^{\alpha +\beta }}{\xi
^{\alpha +\beta }+c^{2}D^{-1}\xi ^{\alpha }t^{\beta }+\left\Vert
\kappa\right\Vert ^{2}c^{2}t^{\alpha +\beta }}\right] \,d\xi,
\end{equation*}%
which gives the representation (\ref{2.18}).
\end{proof}
\begin{example} Let us consider the particular case when $\alpha =\beta $. In this case the
denominator in {\rm (\ref{2.18})} can be written as
\begin{equation*}
\xi ^{2\alpha }+c^{2}D^{-1}t^{\alpha }\xi ^{\alpha }+\left\Vert \kappa\right\Vert
^{2}c^{2}t^{2\alpha }=(\xi ^{\alpha }+A_{+}t^{\alpha })(\xi ^{\alpha
}+A_{-}t^{\alpha }),
\end{equation*}%
where
\begin{equation*}
A_{\pm }=\frac{c^{2}}{2D}(1\pm \Omega ),\qquad \Omega =\sqrt{1-\frac{%
4\left\Vert \kappa\right\Vert ^{2}D^{2}}{c^{2}}}.
\end{equation*}%
We also have
\begin{equation*}
\frac{\xi ^{\alpha }+c^{2}D^{-1}t^{\alpha }}{(\xi ^{\alpha }+A_{+}t^{\alpha
})(\xi ^{\alpha }+A_{-}t^{\alpha })}=\left( \frac{1+\Omega }{2\Omega }%
\right) \frac{1}{\xi ^{\alpha }+A_{-}t^{\alpha }}-\left( \frac{1-\Omega }{%
2\Omega }\right) \frac{1}{\xi ^{\alpha }+A_{+}t^{\alpha }}.
\end{equation*}%
In this case the solution {\rm (\ref{2.18})} reduces to
\[
H(\kappa,t)= \bigg[\left( \frac{1+\Omega }{2\Omega }\right) \frac{1}{2\pi
\mathrm{i}}\int_{\text{Ha}}\text{e}^{\xi }\xi ^{\alpha -1}\frac{1}{\xi
^{\alpha }+A_{-}t^{\alpha }}\,d\xi  -\left( \frac{1-\Omega }{2\Omega }\right) \frac{1}{2\pi \mathrm{i}}\int_{%
\text{Ha}}\text{e}^{\xi }\xi ^{\alpha -1}\frac{1}{\xi ^{\alpha
}+A_{+}t^{\alpha }}\,d\xi \bigg].
\]

The contour integral representation of the one-parameter Mittag-Leffler
function defined as ${\rm{E}}_{\alpha }(z)={\rm{E}}_{\alpha ,1}^{1}(z)$
is
\begin{equation*}
{\rm{E}}_{\alpha }(z)=\frac{1}{2\pi \mathrm{i}}\int_{\text{Ha}}\frac{\text{e}%
^{\xi }\xi ^{\alpha -1}}{\xi ^{\alpha }-z}\,d\xi .
\end{equation*}%
Then $H(\kappa,t)$ can be given by
\begin{equation*}
H(\kappa,t)=\left( \frac{1+\Omega }{2\Omega }\right) {\rm{E}}_{\alpha
}(-A_{-}t^{\alpha })-\left( \frac{1-\Omega }{2\Omega }\right) {\rm{E}}%
_{\alpha }(-A_{+}t^{\alpha }),
\end{equation*}%
which is the solution obtained in {\rm \cite{LV}.}
\end{example}

\begin{remark}\label{rem1} The inverse Fourier transform of {\rm (\ref{2.18})}
has an independent interest.

The Fourier transform $H(\kappa,t)$ is given by {\rm (\ref{2.18})} with initial condition {\rm (\ref{2.12}).} To calculate the inverse Fourier transform  $Q(x,t)$ one has to compute the integral
\begin{equation*}
Q(x,t)=\frac{1}{(2\pi )^{3}}\int_{\mathbb{R}^{3}}\text{e}^{-\mathrm{i}%
 \left\langle\kappa, x\right\rangle}\frac{A}{B+\left\Vert \kappa\right\Vert ^{2}}%
\,d \kappa=\frac{1}{2\pi ^{2}r}\int_{0}^{\infty }\frac{A\mu\sin (\mu r)}{B+\mu^{2}}\,d\mu,
\label{2.24}
\end{equation*}%
where $r=\left\Vert x\right\Vert $ and $\mu=\left\Vert \kappa\right\Vert ^{2}$.
Indeed
\begin{equation*}
q(x,t)=\frac{1}{2\pi \mathrm{i}}\int_{\text{Ha}}\text{e}^{\xi }\xi ^{\alpha
-1}\,(\xi ^{\beta }+c^{2}D^{-1}t^{\beta })\mathcal{F}^{-1}\left[ \frac{1}{%
\xi ^{\alpha +\beta }+c^{2}D^{-1}\xi ^{\alpha }t^{\beta }+\left\Vert
\kappa\right\Vert ^{2}c^{2}t^{\alpha +\beta }}\right] \,d\xi
\end{equation*}%
and
\begin{equation*}
\begin{split}
\mathcal{F}^{-1}\left[ \frac{1}{\xi ^{\alpha +\beta }+c^{2}D^{-1}\xi
^{\alpha }t^{\beta }+\left\Vert \kappa\right\Vert ^{2}c^{2}t^{\alpha +\beta }}%
\right] & =\frac{1}{2\pi ^{2}r}\int_{0}^{\infty }\frac{\mu\sin (\mu r)}{\xi
^{\alpha +\beta }+c^{2}D^{-1}\xi ^{\alpha }t^{\beta }+\mu^{2}c^{2}t^{\alpha
+\beta }}\,dk \\
& =\frac{1}{2\pi ^{2}rc^{2}t^{\alpha +\beta }}\int_{0}^{\infty }\frac{%
\mu^{\prime }\sin {\mu^{\prime }(r/ct^{(\alpha +\beta )/2})}}{\xi ^{\alpha
+\beta }+c^{2}D^{-1}\xi ^{\alpha }t^{\beta }+\mu^{\prime }{}^{2}}\,d\mu^{\prime
}.
\end{split}
\end{equation*}%
This integral can be evaluated using
\begin{equation*}
\int_{0}^{\infty }\frac{\mu\sin (a\mu)}{b^{2}+\mu^{2}}\,dk=\frac{\pi }{2}{\rm{e}}%
^{-ab},
\end{equation*}%
with $a>0$ and ${\rm{Re}}b>0$, which gives
\begin{equation*}
Q(x,t)=\frac{1}{2\pi \mathrm{i}}\frac{1}{4\pi rc^{2}t^{\alpha +\beta }}\int_{%
\text{Ha}}\text{e}^{\xi }\xi ^{\alpha -1}\,(\xi ^{\beta
}+c^{2}D^{-1}t^{\beta }){\displaystyle\text{e}^{-\frac{r\xi ^{\alpha /2}}{%
ct^{(\alpha +\beta )/2}}\sqrt{\xi ^{\beta }+cD^{-1}t^{\beta }}}}\,d\xi .
\end{equation*}
\end{remark}

The function $H(\kappa,t)=H(\left\Vert \kappa\right\Vert ,t),$ $\kappa\in
\mathbb{R}^{3},$ $t\geq 0,$ given by (\ref{2.8}) is radial. We will use the same notation
for function $H(\mu ,t),\mu =\left\Vert \kappa\right\Vert \geq 0,$ $t\geq 0.$

\begin{theorem}\label{th2} Let us assume that for every $t\geq 0$%
\begin{equation}
\int_{0}^{\infty }\mu ^{2}\left\vert H(\mu ,t)\right\vert ^{2}G(d\mu )<\infty
,  \label{2.27}
\end{equation}%
where the finite measure $G$ is given by {\rm (\ref{2.4}).}

Then, the solution $q(x,t)=q(x,t,\omega ),$ $x\in
\mathbb{R}
^{3},$ $\omega \in \mathrm{\Omega },$ $t\geq 0,$ of initial-value problem {\rm (\ref%
{2.2}), (\ref{2.5})} can be written as the stochastic integral%
\begin{equation}
q(x,t)=\int_{%
\mathbb{R}^{3}}e^{\mathrm{i}\left\langle k,x\right\rangle }H(\kappa,t)Z(dk),  \label{2.28}
\end{equation}%
where $H(\kappa,t)$ is given by {\rm (\ref{2.8}),} and the random measure $Z(\cdot)$ is
defined in {\rm (\ref{2.5}).}

The covariance function of the spatio-temporal random field {\rm(\ref{2.28})} is
of the form
\begin{equation}
\mathrm{Cov}(q(x,t),q(x^{\prime },t^{\prime }))=\int_{
\mathbb{R}^{3}}e^{\mathrm{i}\left\langle k,x-x^{\prime }\right\rangle}H(\kappa,t)H(k,t^{\prime })F(dk),
\label{2.29}
\end{equation}%
where the spectral measure $F$ is defined by {\rm (\ref{2.5}).}
\end{theorem}

\begin{proof} If $q$ is defined in (\ref{2.8}), then by (\ref{2.28}) it holds
\begin{eqnarray*}
q(x,t) &=&\int_{%
%TCIMACRO{\U{211d} }%
%BeginExpansion
\mathbb{R}
%EndExpansion
^{3}}\eta (y)Q(x-y,t)=\int_{%
%TCIMACRO{\U{211d} }%
%BeginExpansion
\mathbb{R}
%EndExpansion
^{3}}\eta (x-z)Q(z,t)dz \\
&=&\int_{%
%TCIMACRO{\U{211d} }%
%BeginExpansion
\mathbb{R}
%EndExpansion
^{3}}e^{\mathrm{i}\left\langle k,x\right\rangle }\left[ \int_{%
%TCIMACRO{\U{211d} }%
%BeginExpansion
\mathbb{R}
%EndExpansion
^{3}}e^{\mathrm{i}(k,-z^{\prime })}Q(z,t)dz\right] Z(dk)=\int_{%
%TCIMACRO{\U{211d} }%
%BeginExpansion
\mathbb{R}
%EndExpansion
^{3}}e^{\mathrm{i}\left\langle k,x\right\rangle }H(\kappa,t)Z(dk),
\end{eqnarray*}%
where all stochastic integrals exist in $\mathscr{L}_{2}(\mathrm{\Omega })$
sense.
\end{proof}
\section{Angular Time-Dependent Power Spectrum}\label{sec3}

Consider the unit radius sphere $\mathbf{S}^{2}=\left\{ x\in
\mathbb{R}^{3}:\left\Vert x\right\Vert =1\right\} $ that is centred at the origin with the surface Lebesque measure $\sigma
(dx)=\sigma (d\theta ,d\varphi )=\sin \theta d\theta d\varphi,$ $\theta
\in \lbrack 0,\pi ],$ $\varphi \in \lbrack 0,2\pi ].$

A spatio-temporal spherical random field defined on a probability space $(%
\mathrm{\Omega },\mathcal{F},\mathrm{P})$ is a random
function $T(x,t)=T(x,t,\omega )=T(\theta ,\varphi ,t),$ $x\in \mathbf{S}%
^{2},$ $t\geq 0.$

We consider a real-valued spatio-temporal spherical random field $T$ with
zero mean and finite second-order moment which is continuous in the
mean-square sense, see \cite[p.9]{MP} for definitions and other details.

We assume that the random field $T$ is second-order isotropic, that is $%
\mathrm{E}[T(x,t)T(y,t^{\prime })]=\mathrm{E}[T(\mathbf{g}x,t)T(\mathbf{g}y,t)]$
for every $\mathbf{g}\in SO(3),$ where $SO(3)$ denotes  the group of rotation in $\mathbb{R}
^{3}.$ This is equivalent to the condition that the covariance function $%
\mathrm{E}[T(\theta ,\varphi ,t)T(\theta ^{\prime },\varphi ^{\prime
},t^{\prime })]$ depends only on the angular distance $\gamma =\gamma
_{PP^{\prime }}$ between two points $P=(\theta ,\varphi )$ and $P^{\prime
}=(\theta ^{\prime },\varphi ^{\prime })$ on $\mathbf{S}^{2}$ for every $%
t,t^{\prime }\geq 0.$

Under these conditions, the random field $T\ $can be expanded in  the
mean-square sense as the Laplace series
\begin{equation*}
T(\theta ,\varphi ,t)=\sum_{l=0}^{\infty }\sum_{m=-l}^{l}Y_{lm}(\theta
,\varphi )g_{lm}(t),
\end{equation*}%
where the functions $Y_{lm}(\theta ,\varphi )$ represents the complex
spherical harmonics and the coefficients $g_{lm}(t)$ are complex-valued
stochastic processes defined by
\begin{equation*}
g_{lm}(t)=\int_{0}^{\pi }\int_{0}^{2\pi }T(\theta ,\varphi ,t)\overline{%
Y_{lm}(\theta ,\varphi )}\sin \theta d\theta d\varphi,
\end{equation*}%
and
\begin{equation*}
\mathrm{E}g_{lm}(t)\overline{g_{l^{\prime }m^{\prime }}(t^{\prime })}=\delta
_{ll^{\prime }}\delta _{mm^{\prime }}C_{l}(t,t^{\prime }),\ -l\leq m\leq l,\
-l^{\prime }\leq m^{\prime }\leq l^{\prime },\ l ,l^{\prime }=0,1,2,\ldots .
\end{equation*}

The functional series $\left\{ C_{l}(t,t^{\prime }),\ l=0,1,2,\ldots
\right\} $ is called the angular time-dependent power spectrum of the
isotropic random field $T(\theta ,\varphi ,t)$ on $\mathbf{S}^{2}.$ For
every $t,t^{\prime }\geq 0$ it satisfies the condition
\begin{equation*}
\sum_{l=0}^{\infty }(2l+1)C_{l}(t,t^{\prime })<\infty .
\end{equation*}%

Then, a covariance function between two locations with the angular
distance $\gamma $ at times $t$ and $t^{\prime }$ is equal
\begin{equation*}
R(\cos \gamma ,t,t^{\prime })=\mathrm{E}[T(\theta ,\varphi ,t)T(\theta
^{\prime },\varphi ^{\prime },t^{\prime })]=\frac{1}{4\pi }\sum_{l=0}^{\infty
}(2l+1)C_{l}(t,t^{\prime })P_{l}(\cos \gamma ),
\end{equation*}%
where
\begin{equation*}
P_{l}(x)=\frac{1}{2^{l}\cdot l!}\frac{d^{l}}{dx^{l}}(x^{2}-1)^{2}
\end{equation*}%
is the $l$-th Legendre polynomial.

The random field $q(x,t),$ $x\in
\mathbb{R}^{3},$ $t\geq 0,$ given by (\ref{2.28}) is homogeneous and isotropic in~$x.$ Hence, using the addition theorem for Bessel functions, its covariance
function (\ref{2.29}) can be written in the form%
\begin{eqnarray*}
\mathrm{Cov}(q(x,t),q(x^{\prime },t^{\prime })) &=&\int_{0}^{\infty }\frac{%
\sin (\mu \left\Vert x-x^{\prime }\right\Vert )}{\mu \left\Vert x-x^{\prime
}\right\Vert }H(\mu ,t)H(\mu ,t^{\prime })G(d\mu ) \\
&=&2\pi ^{2}\sum_{l=0}^{\infty }\sum_{m=-l}^{l}Y_{lm}(\theta ,\varphi )%
\overline{Y_{lm}(\theta ^{\prime },\varphi ^{\prime }})\\
&\times& \int_{0}^{\infty }\frac{J_{l+1/2}(\mu r)}{(\mu r)^{1/2}}\cdot \frac{%
J_{l+1/2}(\mu r^{\prime })}{(\mu r^{\prime })^{1/2}}H(\mu ,t)H(\mu
,t^{\prime })G(d\mu ),
\end{eqnarray*}%
where $(r,\theta ,\varphi )$ and $(r^{\prime },\theta ^{\prime },\varphi
^{\prime })$ are spherical coordinates of $x\in
\mathbb{R}
^{3}$ and $y\in
\mathbb{R}
^{3}$ respectively.

Using the Karhunen theorem \cite[p.10]{L} one obtains the following spectral
representation of homogeneous and isotropic random field in $%
%TCIMACRO{\U{211d} }%
%BeginExpansion
\mathbb{R}
%EndExpansion
^{3}:$%
\begin{equation}
q(x,t)=q(r,\theta ,\varphi ,t)=\pi \sqrt{2}\sum_{l=0}^{\infty
}\sum_{m=-l}^{l}Y_{lm}(\theta ,\varphi )\int_{0}^{\infty }\frac{%
J_{l+1/2}(\mu r)}{(\mu r)^{1/2}}H(\mu ,t)Z_{lm}(d\mu ),  \label{3.6}
\end{equation}%
where the random measures $Z_{lm}(\cdot)$ are defined in (\ref{2.6}).

The restriction of the homogeneous and isotropic random field (\ref{3.6}) to
the sphere $\mathbf{S}^{2}$ is an isotropic spherical random field $%
T_{H}(x,t),$ $x\in \mathbf{S}^{2},$ $t\geq 0,$ which will be called a spherical fractional
hyperbolic diffusion random field (SFHDRF).

In this case
\begin{equation}\label{cov}
\mathrm{Cov}(T_{H}(x,t),T_{H}(x^{\prime },t^{\prime }))=R(\cos \gamma
,t,t^{\prime })=\int_{0}^{\infty }\frac{\sin (2\mu \sin  \left(\frac{\gamma }{2}\right))}{%
2\sin \left(\frac{\gamma }{2}\right)}H(\mu ,t)H(\mu ,t^{\prime })G(d\mu ),
\end{equation}%
where the Euclidean distance $\left\Vert x-x^{\prime }\right\Vert $ (also called
the chordal distance between two points $x,x^{\prime }\in \mathbf{S}%
^{2}\subset
\mathbb{R}
^{3}$) can be expressed in terms of the great circle (also known as geodesic
or spherical) distance as follows: $\left\Vert x-x^{\prime }\right\Vert
=2\sin (\gamma /2),$ $\gamma =\gamma (x,x^{\prime })=\arccos(
\left\langle x,x^{\prime }\right\rangle ).$

By addition theorem for spherical Bessel functions, the isotropic random
field $T_{H}(x,t),$ $x\in \mathbf{S}^{2},$ $t>0,$ has the following spectral
representation:%
\begin{equation}
T_{H}(x,t)=\sum_{l=0}^{\infty }\sum_{m=-l}^{l}Y_{lm}(\theta ,\varphi
)a_{lm}^{H}(t),  \label{3.8}
\end{equation}%
with the stochastic processes
\begin{equation}\label{alm}
a_{lm}^{H}(t)=\pi \sqrt{2}\int_{0}^{\infty }\frac{J_{l+1/2}(\mu r)}{(\mu
r)^{1/2}}H(\mu ,t)Z_{lm}(d\mu ),\ t\geq 0.
\end{equation}

Thus, the angular spectrum of the SFHDRF in (\ref{3.8}) is of the form%
\begin{equation}
C_{l}(t,t^{\prime })=\mathrm{E}a_{lm}^{H}(t)\overline{a_{lm}^{H}(t^{\prime })%
}=2\pi ^{2}\int_{0}^{\infty }\frac{J^2_{l+1/2}(\mu )}{\mu }H(\mu ,t)H(\mu ,t')G(d\mu ),\
l=0,1,2,\ldots ,  \label{3.9}
\end{equation}%
where%
\begin{equation} \label{3.10}
    \begin{aligned}
H(\mu ,t) &=1+\sum_{n=0}^{\infty }\left( -\mu ^{2}c^{2}t^{\alpha +\beta
}\right) ^{n+1}{\rm{E}}_{\beta ,(\alpha +\beta )(n+1)-1}^{n+1}\left( -\frac{%
c^{2}}{D}t^{\beta }\right)  \\
 &=1-\mu ^{2}c^{2}t^{\alpha +\beta }\sum_{m=0}^{\infty }\sum_{n=0}^{\infty }%
\binom{m}{n}\frac{\left( -\frac{c^{2}}{D}t^{\beta }\right) ^{m}(\mu
^{2}Dt^{\alpha })^{n}}{\Gamma (\beta m+\alpha n+\alpha +\beta +1)}.
\end{aligned}
\end{equation}

We can summarise the above results as the following theorem, which is the main result of the paper.

\begin{theorem}\label{th3}  Under \ the condition {\rm (\ref{2.27})} the SFHDRF in {\rm (\ref{3.8})} is an isotropic random field on $\mathbf{S}^{2}$ with the angular power spectrum given
by formulae {\rm (\ref{3.9})} and {\rm (\ref{3.10}).}
\end{theorem}

\begin{example}\label{ex2} Consider the Mat\'{e}rn class of covariance functions for the initial condition {\rm (\ref{2.4})} which has the following form:
\begin{equation*}
\mathrm{E}\eta (x)\eta (x^{\prime })=\sigma ^{2}\frac{2^{1-\nu }}{\Gamma
(\nu )}\left( a\left\Vert x-x^{\prime }\right\Vert \right) ^{\nu }K_{\nu
}(a\left\Vert x-x^{\prime }\right\Vert ),\ x,x^{\prime }\in
\mathbb{R}
^{3},
\end{equation*}%
where $\sigma ^{2}>0,$ $a>0,$  and $K_{\nu }$ is the modified Bessel
function of the second kind of order~$\nu >0.$  This class found numerous applications in statistics and machine learning, see, for example, {\rm \cite{LM}, \cite{Por},} and references therein.

Then, the spectral measure $G$
in {\rm (\ref{2.4})} has the isotropic spectral density $g(\mu )$ defined by  $G^{\prime }(\mu
)=4\pi \mu ^{2}g(\mu ),$ which is equal
\begin{equation*}
g(\mu )=\frac{\sigma ^{2}\Gamma (\nu +\frac{3}{2})a^{2\nu }}{\pi
^{3/2}\Gamma (\nu )}\frac{1}{(a^{2}+\mu ^{2})^{\nu +\frac{3}{2}}},\ \mu \geq
0.
\end{equation*}

Here, the parameter $\nu $ controls the degree of differentiability of
random field $\eta (x),x\in
%TCIMACRO{\U{211d} }%
%BeginExpansion
\mathbb{R}
%EndExpansion
^{3},$ which represent the initial condition {\rm (\ref{2.4}),} $\sigma $ is
field's variance and the parameter $a$ is the scale parameter.

Thus, condition {\rm (\ref{2.27})} is equivalent to
\begin{equation*}
\int_{0}^{\infty }\frac{\mu ^{2}}{(a^{2}+\mu ^{2})^{\nu +\frac{3}{2}}}%
\left\vert H(\mu ,t)\right\vert ^{2}d\mu <\infty ,
\end{equation*}%
for every $t\geq 0.$ It is satisfied for $\alpha =\beta \in (0,1],$
see {\rm \cite{LV}} for details, and for general $1<\alpha +\beta \leq 2$
for $\nu >0.$

Let us consider the partial case of the initial-value problem {\rm (\ref{2.2}), (\ref%
{2.3}),} when $\alpha =\beta \in (0,1].$ Then, see Remark~{\rm\ref{rem1},}
\begin{equation*}
H(\mu ,t)=\left( \frac{1+\Omega }{2\Omega }\right) {\rm{E}}_{\alpha
}(-A_{-}t^{\alpha })+\left( \frac{1-\Omega }{2\Omega }\right) {\rm{E}}%
_{\alpha }(-A_{+}t^{\alpha })=
\end{equation*}%
\begin{equation*}
=\frac{1}{2}({\rm{E}}_{\alpha }(-A_{-}t^{\alpha })+{\rm{E}}_{\alpha
}(-A_{+}t^{\alpha }))+\frac{1}{2\Omega }({\rm{E}}_{\alpha }(-A_{-}t^{\alpha
})-{\rm{E}}_{\alpha }(-A_{+}t^{\alpha })),
\end{equation*}%
where
\begin{equation*}
A_{\pm }=A_{\pm }(\mu )=\frac{c^{2}}{2D}(1\pm \Omega ),
\end{equation*}%
\begin{equation*}
\Omega =\Omega (\mu )=\sqrt{1-\frac{4\mu ^{2}D^{2}}{c^{2}}}=\sqrt{1-\frac{%
4\mu ^{2}D^{2}}{c^{2}}}\,\mathbf{1}_{\mu \leq c/2D}+\mathrm{i}\sqrt{\frac{4\mu
^{2}D^{2}}{c^{2}}-1}\,\mathbf{1}_{\mu >c/2D},
\end{equation*}%
and $\mathbf{1}_{\left\{ {}\right\} }$ denotes the indicator function.

This case was considered in {\rm \cite{LV}}. The case $\alpha =1,$ $\alpha +\beta
=2 $ is also known, see {\rm \cite{BKLO1}, \cite{K}, \cite{KR},} since ${\rm{E}}%
_{1}(z)=e^{z}$ , and then
\begin{equation*}
{\rm{E}}_{1}(A_{\pm })=\exp \left\{ -\frac{c^{2}t}{2D}(1\pm \Omega )\right\}
,
\end{equation*}%
which gives%
\begin{eqnarray*}
H(\mu ,t) &=&\exp \left\{ -\frac{c^{2}t}{2D}\right\} \left\{\left[ \cosh \left( ct%
\sqrt{\frac{c^{2}}{4D^{2}}-\mu ^{2}}\right) +\frac{c}{2D\sqrt{\frac{c^{2}}{%
4D^{2}}-\mu ^{2}}}\sinh \left( ct\sqrt{\frac{c^{2}}{4D^{2}}-\mu ^{2}}\right) %
\right] \mathbf{1}_{\mu \leq c/2D} \right.\\
&+&\left.\left[ \cos \left( ct\sqrt{\mu ^{2}-\frac{c^{2}}{4D^{2}}}\right) +\frac{c%
}{2D\sqrt{\mu ^{2}-\frac{c^{2}}{4D^{2}}}}\sin \left( ct\sqrt{\mu ^{2}-\frac{%
c^{2}}{4D^{2}}}\right) \right] \mathbf{1}_{\mu >c/2D}\right\},
\end{eqnarray*}%
and the angular spectrum
\begin{equation*}
C_{l}(t,t^{\prime })=2\pi ^{2}\left[ \int_{0}^{\frac{c}{2D}}\frac{%
J_{l+1/2}^{2}(\mu )}{\mu }\bar{H}_{1}(\mu ,t)\tilde{H}_{1}(\mu ,t^{\prime
})G(d\mu )+\int_{\frac{c}{2D}}^{\infty }\frac{J_{l+1/2}^{2}(\mu )}{\mu }\bar{%
H}_{2}(\mu ,t)\tilde{H}_{2}(\mu ,t^{\prime })G(d\mu )\right] ,
\end{equation*}%
where
\begin{eqnarray*}
\bar{H}_{1}(\mu ,t) &=&\exp \left\{ -\frac{c^{2}t}{2D}\right\} \left[ \cosh
\left( ct\sqrt{\frac{c^{2}}{4D^{2}}-\mu ^{2}}\right) +\frac{c}{2D\sqrt{\frac{%
c^{2}}{4D^{2}}-\mu ^{2}}}\sinh \left( ct\sqrt{\frac{c^{2}}{4D^{2}}-\mu ^{2}}%
\right) \right] \mathbf{1}_{\mu \leq c/2D}, \\
\bar{H}_{2}(\mu ,t) &=&\exp \left\{ -\frac{c^{2}t}{2D}\right\} \left[ \cos
\left( ct\sqrt{\mu ^{2}-\frac{c^{2}}{4D^{2}}}\right) +\frac{c}{2D\sqrt{\mu
^{2}-\frac{c^{2}}{4D^{2}}}}\sin \left( ct\sqrt{\mu ^{2}-\frac{c^{2}}{4D^{2}}}%
\right) \right] \mathbf{1}_{\mu >c/2D}.
\end{eqnarray*}
\end{example}
\section{Numerical studies}\label{sec4}
This section  numerically investigates  the solution $T_H( x,t)$ and its spectral and covariance properties with respect to parameters $\alpha$ and $\beta.$  Two examples similar to those presented in~\cite{BKLO1} (where the case of $\alpha=\beta =1$ was considered)  are used to study the impact of $\alpha$ and $\beta.$

We use the approach developed in \cite{BKLO1}, see the justifications and detailed discussions there. However,  the numerical analysis in this paper requires more  sophisticated approximation methods
compared to \cite{BKLO1}. Namely,  a much more general case of two fractional derivatives of the orders $\alpha$ and $\beta$ is considered. Thus, in this general case,  the function $H(\mu,t)$ is not given via elementary functions and for computations one has to use a truncated version of the series (\ref{3.10}). Figure~\ref{fig1_0} uses the truncated double series with 80 $m$ and $n$ terms to illustrate the dependence of $H(\mu,t)$ on its parameters. The first subplot in Figure~\ref{fig1_0} shows $H(1,0.1)$ for $\alpha,\beta \in [0.5,1].$  The second subplot illustrates changes of $H(\mu,t)$ with respect to its arguments $\mu\in [1,20]$ and $t\in[0,1]$ for the fixed values of parameters $\alpha =0.8$ and $\beta=1.$  In the bottom row  $H(1,t)$ is shown as a functions of $t$ and $\alpha$ or $\beta$ for fixed $\beta=1$ and $\alpha=0.5$ respectively.
The plots illustrate that, as a general trend, the functions $H(\mu,t)$ exhibit a decreasing behaviour as their parameters $\alpha$ and $\beta$ decrease. Furthermore, there is a decline of $H(\mu,t)$ with increasing values of time $t$. Increasing the value of $\mu$ introduces a decaying oscillatory behaviour.
   \begin{figure}[!htb]
    \centering
    \subfigure[$H(1,0.1)$ as a function of $\alpha$ and $\beta$]{\includegraphics[trim = 3.5cm 1.5cm 1.cm 2.5cm,clip,width=0.48\linewidth,height=0.4\linewidth]{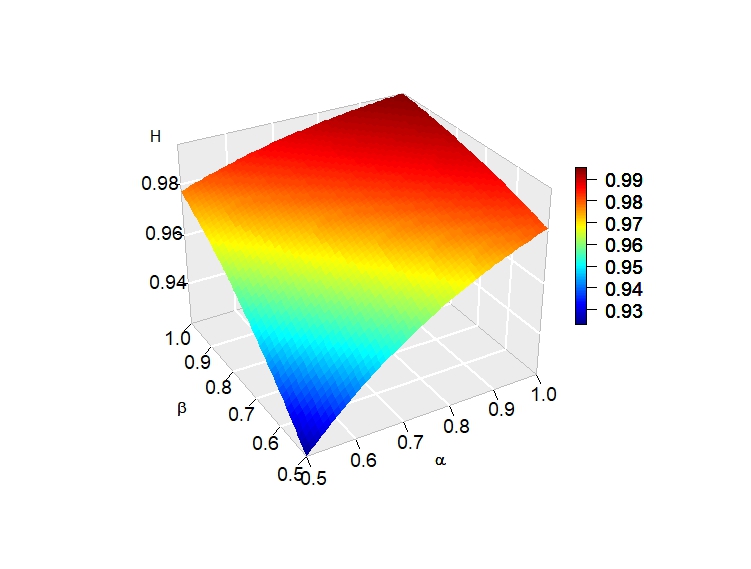}}
    \subfigure[$H(\mu,t)$ for fixed $\alpha=0.8$ and $\beta=1$]{	\includegraphics[trim =2.5cm 2cm 2.5cm 3cm,width=0.49\linewidth,height=0.38\linewidth]{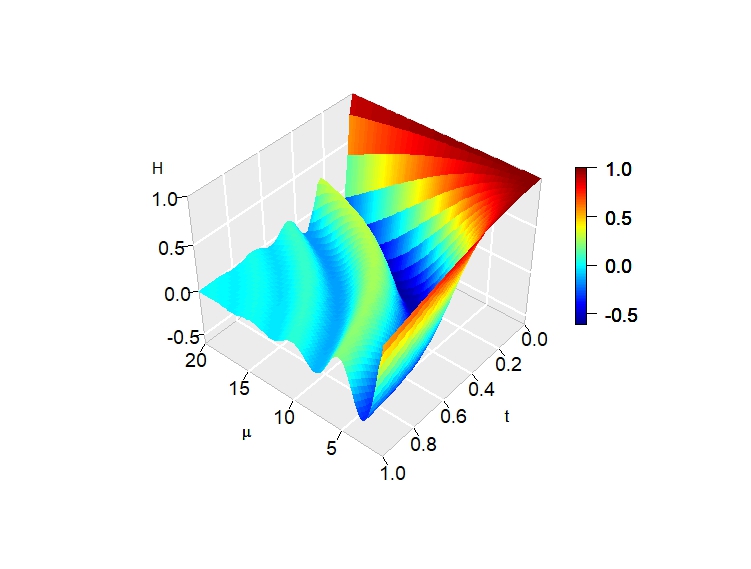}}\vspace{2cm}
    \subfigure[$H(1,t)$ as a function of $t$ and $\alpha$ when $\beta=1$]{	\includegraphics[trim = 1cm 2cm 2.5cm 3.5cm,width=0.48\linewidth,height=0.34\linewidth]{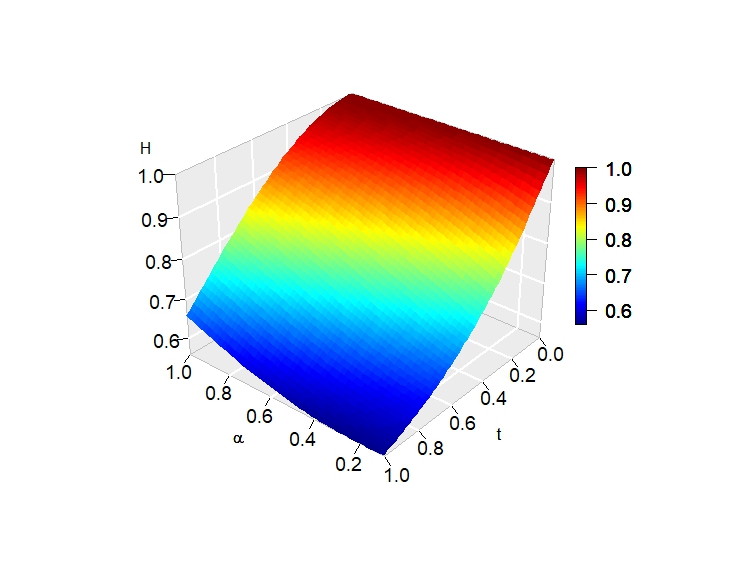}}
        \subfigure[$H(1,t)$ as a function of $t$ and $\beta$ when $\alpha=0.5$]{	\includegraphics[trim =1cm 2cm 2.5cm 3.5cm,width=0.49\linewidth,height=0.34\linewidth]{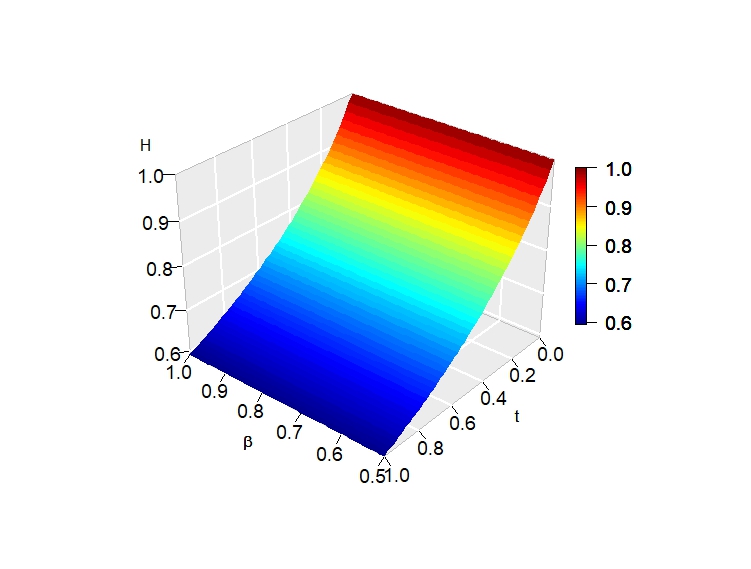}}
    \caption{Dependence of $H(\mu,t)$ on its parameters and arguments} \label{fig1_0}
\end{figure}

 Note that, similar  to \cite{BKLO1}, there are no explicit elementary functional relations between $C_l(t,t')$, $R(\cos \gamma, t,t')$ and $C_l(0,0)$, $R(\cos \gamma, 0,0)$ respectively. To compute spectral and covariance functions of $T_H(x,t)$ at time $t>0$ one has to approximate the integral representations (\ref{cov}), (\ref{alm}) and (\ref{3.9}) that use the spectral measure $G(\cdot)$ and stochastic measures $Z_{lm}(\cdot)$ of the initial random condition field $\eta( x)$.

In the following examples we illustrate the obtained results by using simulated data with the covariance function and oscillating angular spectrum that are similar to those observed for the CMB data, see plots and discussions in ~\cite{BKLO1}. In the  examples we assume that the measure $G(\cdot)$ is discrete with a finite support set. Absolutely continuous spectral measures  can be approximated by selecting a sufficiently a large dense support set.

Let us denote by $\{\mu_i,\ i=1,\dots,I\}$ the support of $G(\cdot)$  and its values by
$$\sigma_i^2=G(\mu_i)=\mathbf  E\ Z_{lm}^2(\mu_i),\ i=1,\dots,I,$$
where  $Z_{lm}(\cdot)$ are real-valued random variables.
Under the assumptions that the random field $\eta(x)$ is centered Gaussian, the random variables  $Z_{lm}(\mu_i)\sim N(0, \sigma_i^2)$ and independent for different $l$, $m$ and $i$.

Then, formulae (\ref{cov}), (\ref{alm}) and (\ref{3.9}) take the following discrete forms
\begin{equation}\label{Disc}
    R(\cos \gamma, t,t')=\sum_{i=1}^{I} \frac{\sin(2\mu_i\sin(\frac{\gamma}{2}))}{2\mu_i\sin(\frac{\gamma}{2})}\tilde{H}(\mu_i,t)\tilde{H}(\mu_i,t')\sigma_i^2,
\end{equation}
$$a_{lm}(t)=\pi \sqrt{2} \sum_{i=1}^{I} \frac{J_{l+\frac{1}{2}}(\mu_i)}{\sqrt{\mu_i}}\tilde{H}(\mu_i,t)Z_{lm}(\mu_i),$$
\begin{align}\label{simCl}
    C_l(t,t')=2\pi^2 \sum_{i=1}^{I} \frac{J_{l+\frac{1}{2}}^2(\mu_i)}{\mu_i}\tilde{H}(\mu_i,t)\tilde{H}(\mu_i,t')\sigma_i^2.
\end{align}

\begin{example}
    This example illustrates changes over time of the covariance function $R(\cos{\gamma},t,t)$ and the power spectrum $C_{l}(t,t).$  The  parameters $c=1$ and $D=1$ in equation {\rm (\ref{2.2})} were selected.  To produce plots and computations we used the corresponding discrete equations {\rm (\ref{Disc})} and {\rm (\ref{simCl})} with the spectrum support  $\mu_i=1+4(i-1)$ and values $\sigma_{i}={100}/{i},$  $i\in\{1,2,\dots,10\},$ i.e. $I=10.$  To compute approximate values of the functions $H(\mu,t)$ we used truncated double series in {\rm (\ref{3.10})} with 80 terms for both $n$ and $m.$ The empirical studies show that increasing the number of terms does not change the plots.

   For the cases of $\alpha=\beta=1$ and $\alpha=0.8,$ $\beta=1,$ Figure~{\rm \ref{fig:1_1}} show the covariance $R(\cos\gamma, t,t)$ at the time lags $t=0,\ t=0.1$ and $t=0.5$ as functions of the angular distance $\gamma.$ As expected, the first subplot in Figure~{\rm\ref{fig:1_1}} coincides with the particular case considered in Example~{\rm\ref{ex2}} and the corresponding plot in {\rm\cite[Figure 2a]{BKLO1}. } The covariance plots in Figure~{\rm\ref{fig:1_1}} and results for the same $t=0.1$ but different values $\alpha$ and $\beta$ in Figure~{\rm\ref{fig:1_2}} suggest that the variance of the field and dependencies at short angular distances decrease with decreasing values of $\alpha$  and $\beta.$ At large angular distances dependencies are smaller and therefore their decrease has a smaller magnitude.

   \begin{figure}[!htb]
       \centering
       \subfigure[The case of $\alpha=\beta=1$]{\includegraphics[trim = 0cm 0cm 0cm 2.0cm,clip,width=0.49\linewidth,height=0.4\linewidth]{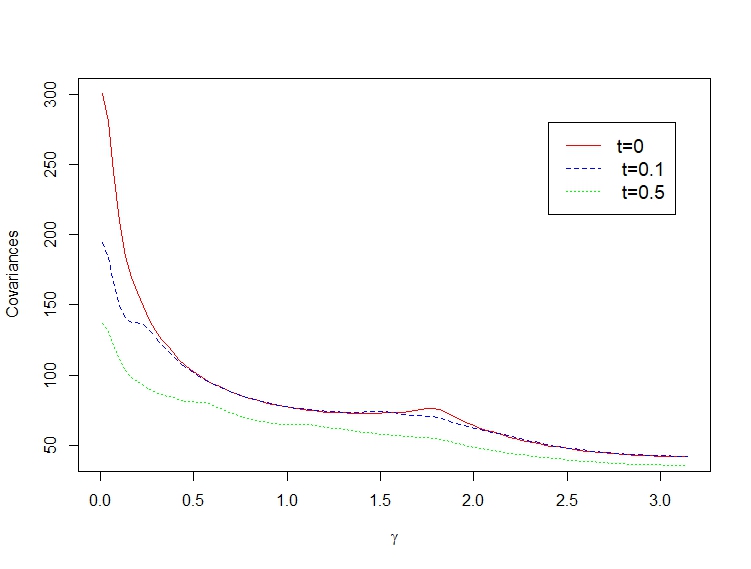}}
       \subfigure[The case of $\alpha=0.8$ and $\beta=1$ ]{\label{fig1_1b}	\includegraphics[trim = 0cm 0cm 0cm 2.0cm,width=0.49\linewidth,height=0.4\linewidth]{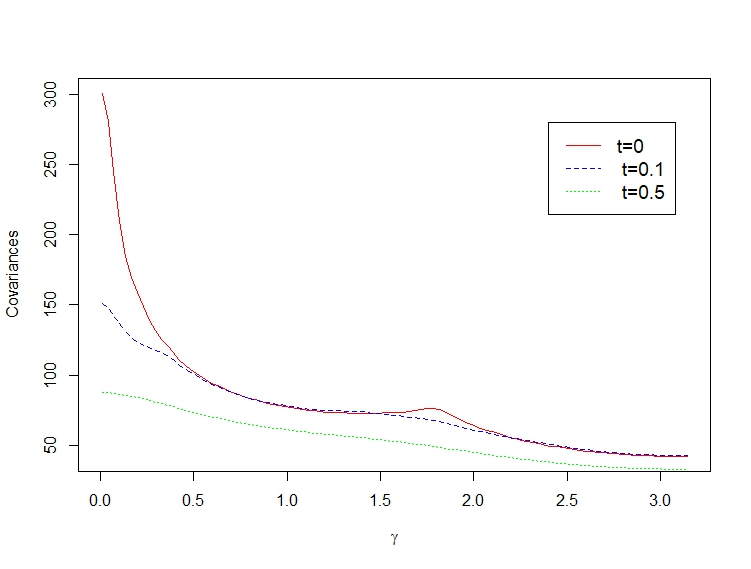}}
       \caption{$R(\cos \gamma,t,t)$ at the time lags $t=0,\ 0.1$, and $0.5$  for $c=D=1$} \label{fig:1_1}
   \end{figure}

      \begin{figure}[!htb]
       \centering
       \subfigure[The case of $\beta=1$]{\includegraphics[trim = 0cm 0cm 0cm 1.0cm,clip,width=0.49\linewidth,height=0.4\linewidth]{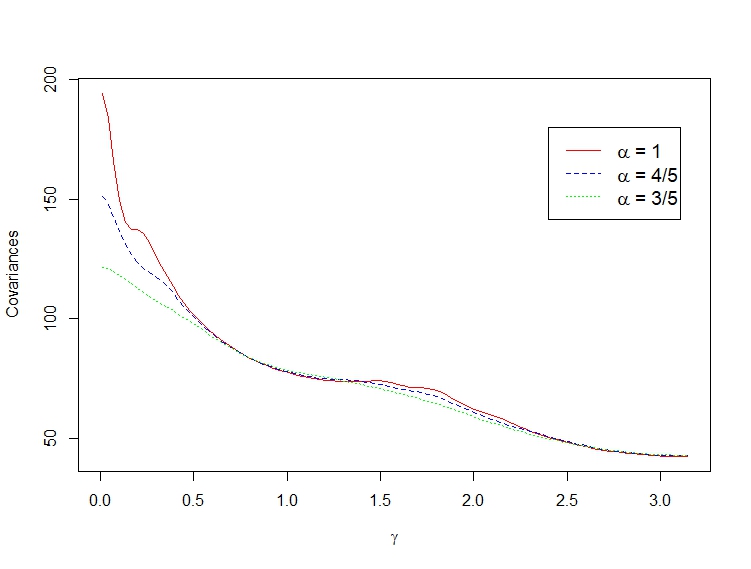}}
       \subfigure[The case of $\alpha=1$ ]{	\includegraphics[trim = 0cm 0cm 0cm 1.0cm,width=0.49\linewidth,height=0.4\linewidth]{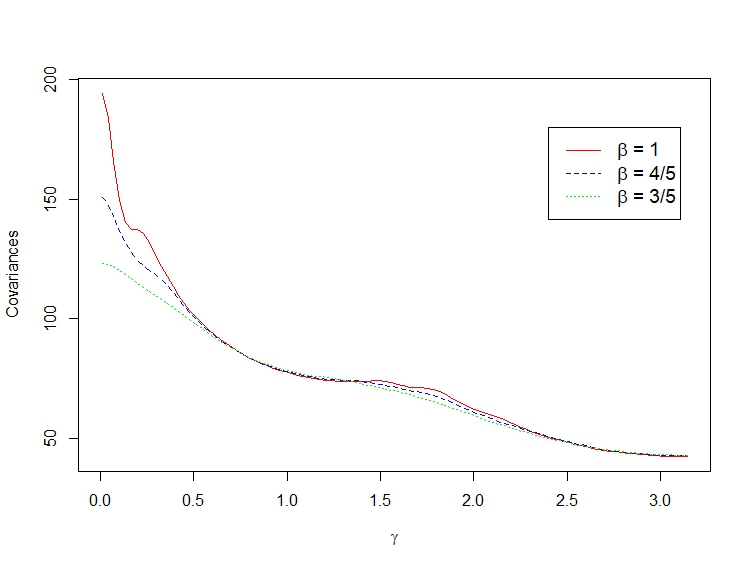}}
       \caption{$R(\cos \gamma,t,t)$ at the time lag $t=0.1$  for $c=D=1$} \label{fig:1_2}
   \end{figure}

 To understand effects of the orders of fractional derivatives  and the angular distance on the covariance function, we provide 3D-plots in Figure~{\rm \ref{fig1_3}}. They show the covariance as functions of the angular distances~$\gamma$ and the parameters $\alpha$ and $\beta.$  The fixed numeric values $\beta=1$ and $\alpha=0.8$ of another fractional order were used for each of the subplots respectively. The values $c=D=1$ and $t=0.1$ were selected.  The plots in Figure~{\rm\ref{fig1_3}} were normalized by the maximum value of $R(\cos\gamma,0.1,0.1).$  The plots suggest that the covariance decays when the angular distance increases and the fractional orders decrease. Moreover, it decreases faster with respect to decreasing  $\alpha$ than $\beta.$
    \begin{figure}[!htb]
        \centering
        \subfigure[$R(\cos \gamma,0.1,0.1)$ as a function of $\gamma$ and $\alpha$ for $\beta=1$]{\includegraphics[trim = 2.0cm 2cm 2.9cm 2.2cm,clip,width=0.50\linewidth,height=0.41\linewidth]{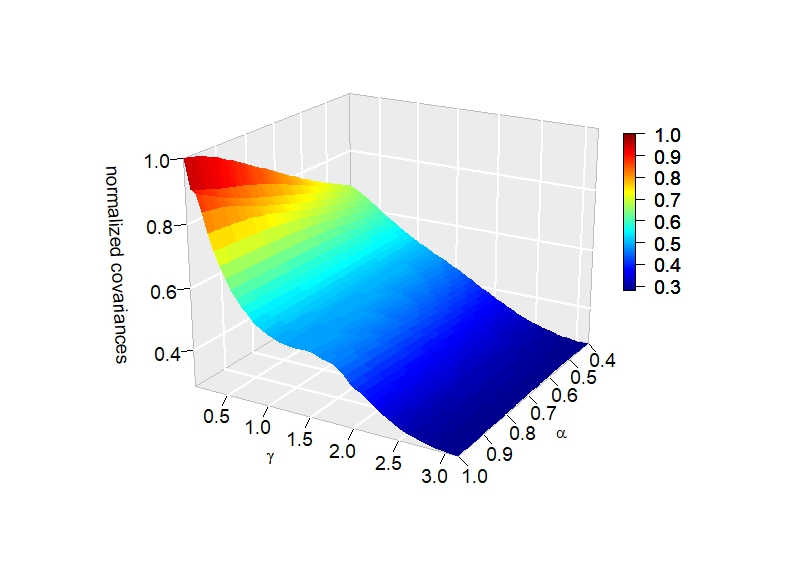}}\hspace{0.1 cm}
        \subfigure[$R(\cos \gamma,0.1,0.1)$ as a function of $\gamma$ and $\beta$ for $\alpha=0.8$]{\label{fig:b1}\includegraphics[trim = 2.0cm 2cm 2.9cm 2.2cm,clip,width=0.48\linewidth,height=0.41\linewidth]{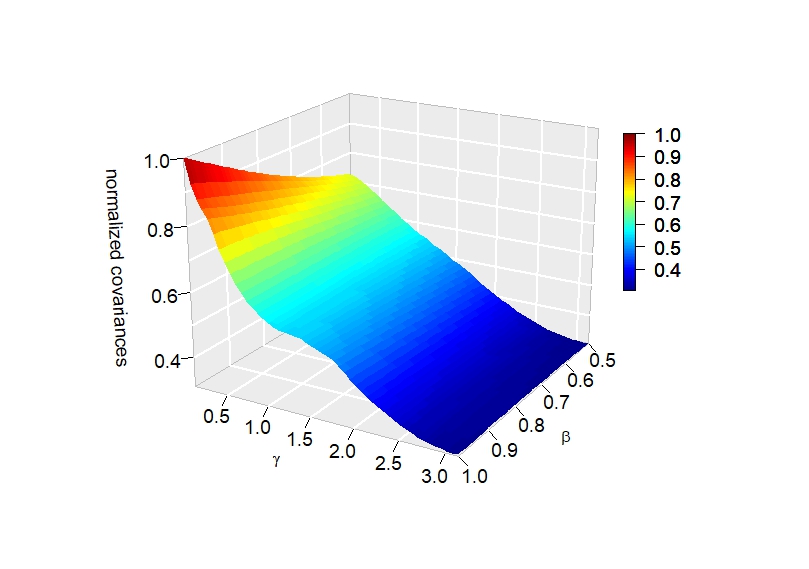}}
        \caption{$R(\cos \gamma,t,t)$ at the time lag $t=0.1$  for $c=D=1$ }\label{fig1_3}
    \end{figure}

Figure {\rm \ref{fig1_4}} displays the power spectrum $C_{l}(t,t)$ as a function of time $t$ and the fractional orders $\alpha$ and $\beta.$ To produce this figure we used the time interval $t\in[0,0.5]$ and two sets of indices $l=2, 5, 10, 20$ and $l=3, 6, 11, 21.$  Two combinations of orders of fractional derivatives,  $\alpha =0.8,\ \beta=1$ and $\alpha =1,\ \beta=0.8,$  were used. The plots demonstrate that the power spectrum magnitudes  decay very quickly  regardless of the values of $\alpha$ and $\beta$ when $l$ increases.
    \begin{figure}[!htb]
    \centering
    \subfigure[$C_{l}(t,t)$ as a function of $t$  for  $\alpha =0.8$ and $\beta=1$]{\label{fig:a1}\includegraphics[trim = 0.0cm 0cm 0.9cm 2.0cm,clip,width=0.48\linewidth,height=0.41\linewidth]{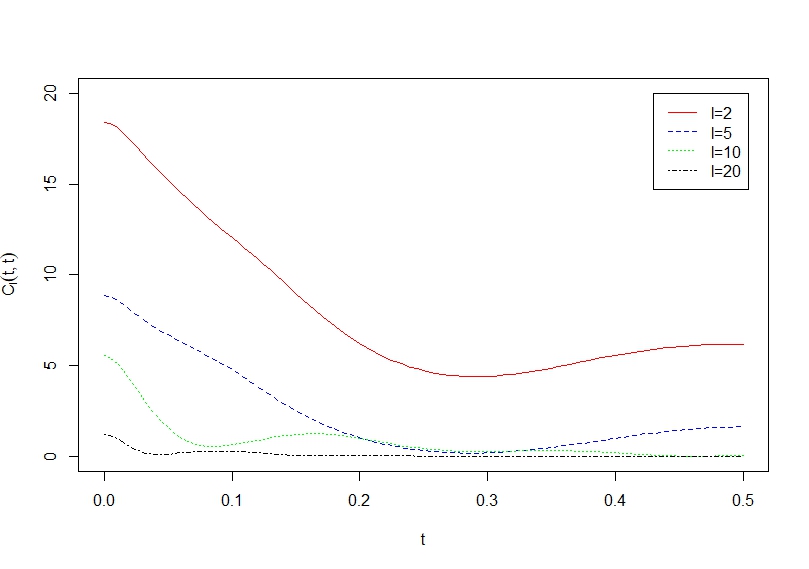}}\hspace{0.1 cm}
    \subfigure[$C_{l}(t,t)$ as a function of $t$  for  $\alpha=1$ and  $\beta=0.8$ ]{\label{fig:b1}\includegraphics[trim = 0.0cm 0cm 0.9cm 2.0cm,clip,width=0.48\linewidth,height=0.41\linewidth]{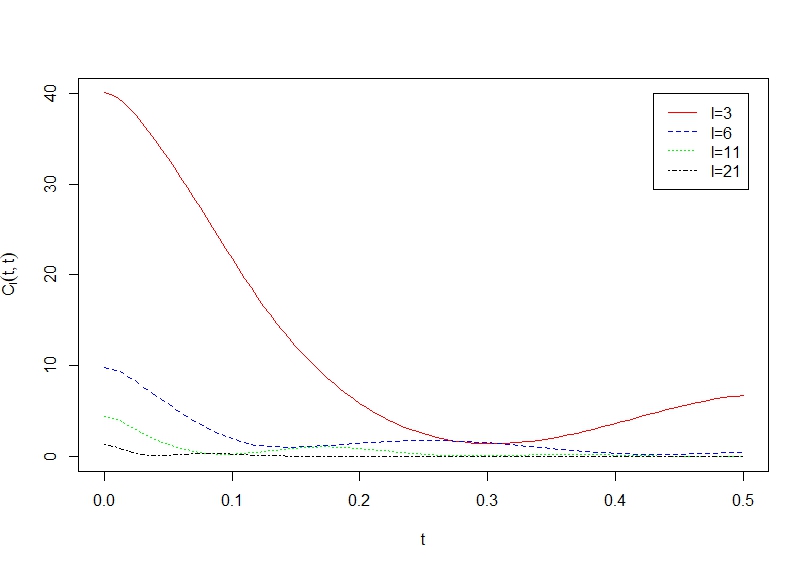}}
    \caption{$C_{l}(t,t)$ on the interval $[0,0.5]$ for $c=D=1$ }\label{fig1_4}
\end{figure}
\end{example}
\begin{example}
In this numeric example we consider SFHDRF for the parameters $c=1$ and $D=2.$
The initial condition random field $\eta(\mathbf  x)$ with low and high frequency components is studied.  Namely, the support of its discrete spectrum belongs to the intervals $[0,20]$ and $[80,90].$ To simulate realisations of random fields at the same small scale as real CMB values we use the values $\sigma_i^2=0.00003$ and $0.0001$ for low and high frequency components respectively.

We use the CMB tools and colour palettes from the R package rcosmo~{\rm\cite{fryer2018rcosmo}} and the Python package healpy  for fast computations of Healpix images from the corresponding Laplace series.
  As approximations of spherical fields we use  Laplace series with the first $100$ coefficients $C_l$ obtained by applying~{\rm (\ref{simCl})} to the above discrete spectrum. For the purpose of comparison, the estimated angular spectrum of the actual CMB data is presented in the first subplot of  Figure~{\rm \ref{fig2_1}}. The second subplot gives the angular power spectrum coefficients $C_l(t,t)$  at the specific time instances $t=0, 0.05$ and $0.1$ for different values of $\alpha$ and $\beta.$   The example demonstrates  that simple discrete spectral distribution can effectively generate a pattern that resembles the genuine angular power spectrum. The magnitude of  $C_l(t,t)$  in the second subplot of Figure~{\rm \ref{fig2_1}} exhibits a decreasing trend as time $t$ progresses and as the values of $\alpha$ and $\beta$ decrease.
       \begin{figure}[!htb]
       \centering
       \subfigure[Angular power spectrum of CMB data]{\includegraphics[trim = 1.2cm 0cm 0.9cm 2.0cm,clip,width=0.46\linewidth,height=0.41\linewidth]{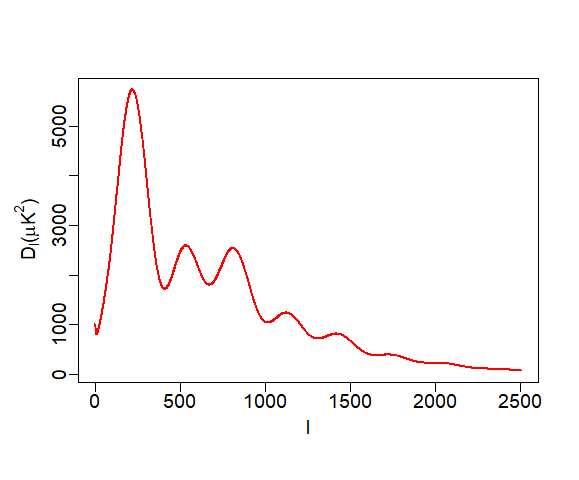}}\hspace{0.1 cm}
       \subfigure[$C_{l}(t,t)$ for $c=1$ and $D=2.$]{\includegraphics[trim = 0.0cm 0cm 0.9cm 2.0cm,clip,width=0.5\linewidth,height=0.41\linewidth]{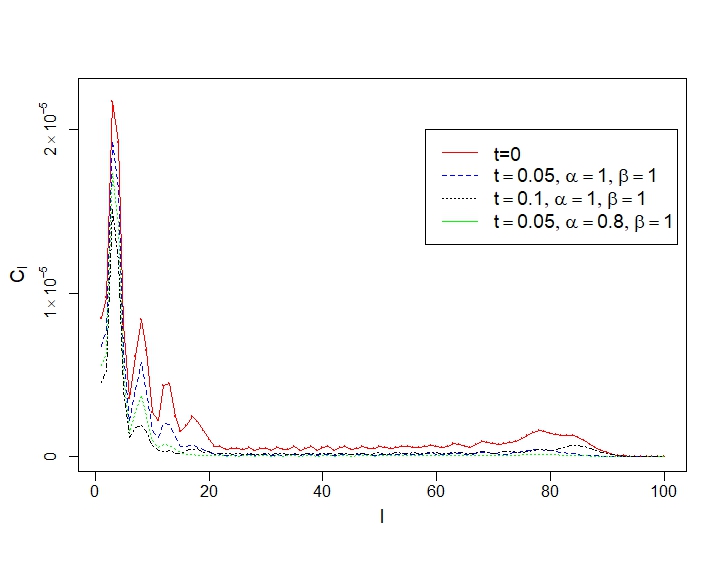}}
       \caption{Angular power spectrum}\label{fig2_1}
   \end{figure}

  The first map  in Figure~{\rm \ref{fig2_2}} shows the initial condition random field  i.e. $T_{H}(x,0)$ computed for $t=0.$
Then,    maps are presented for $t=0.05$ and various combinations of parameters $\alpha$ and $\beta.$ The obtained plots suggest that with increasing $t$ the field is getting smother, which is expected from general physical and cosmological arguments. Also, it is evident from the maps that smaller values of the orders of fractional derivatives make temporal fields smoother.
    \begin{figure}[!htb]
    \centering
    \subfigure[The case of $t=0,$   $\alpha =1$ and $\beta=1$]{\includegraphics[trim = 0.5cm 2cm 1.5cm 3.0cm,clip,width=0.48\linewidth,height=0.41\linewidth]{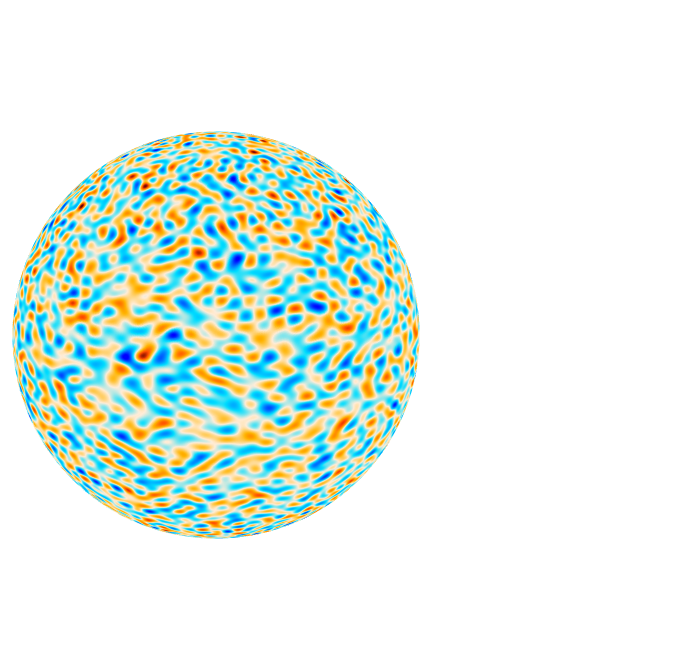}} \hspace{-3 cm}
    \subfigure{\includegraphics[trim = 2.0cm 2cm 0.9cm 2.0cm,clip,width=0.2\linewidth,height=0.41\linewidth]{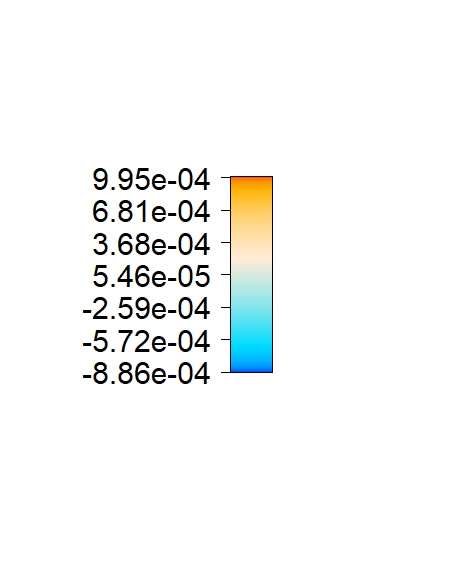}}
    \hspace{-1 cm}
    \subfigure[The case of $t=0.05,$   $\alpha =1$ and $\beta=1$]{\includegraphics[trim = 0.5cm 2cm 1.5cm 3.0cm,clip,width=0.48\linewidth,height=0.41\linewidth]{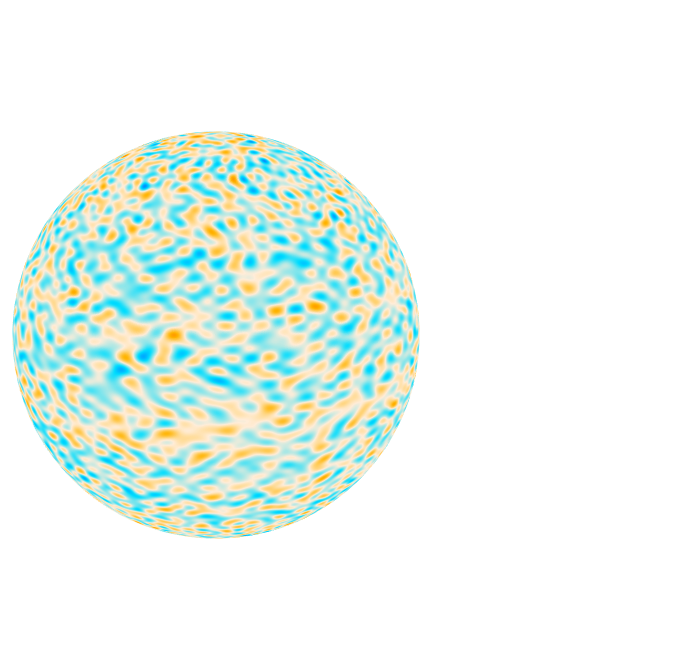}} \hspace{-3 cm}
\subfigure{\includegraphics[trim = 2.0cm 2cm 0.9cm 2.0cm,clip,width=0.2\linewidth,height=0.41\linewidth]{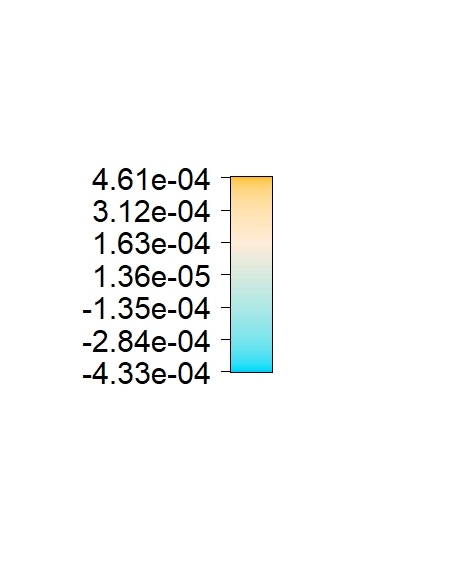}}
    \subfigure[The case of $t=0.05,$   $\alpha =0.8$ and $\beta=1$]{\includegraphics[trim = 0.5cm 2cm 1.5cm 3.0cm,clip,width=0.48\linewidth,height=0.41\linewidth]{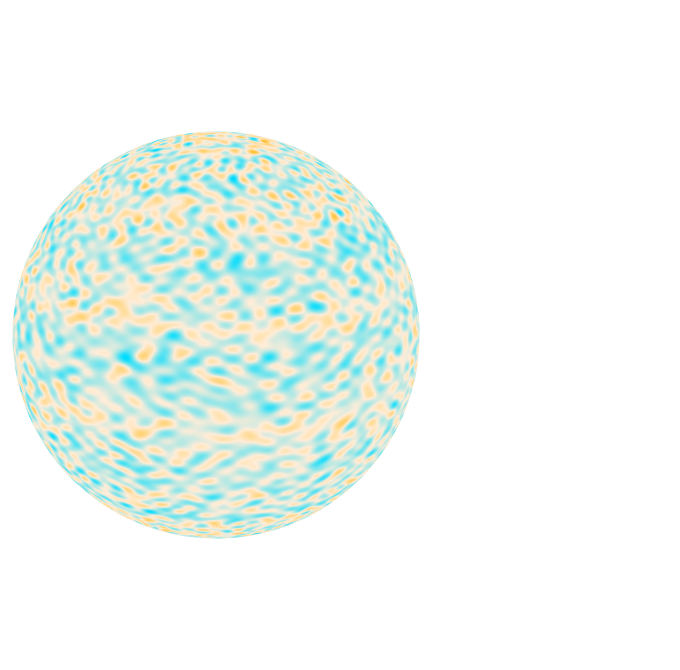}} \hspace{-3 cm}
\subfigure{\includegraphics[trim = 2.0cm 2cm 0.9cm 2.0cm,clip,width=0.2\linewidth,height=0.41\linewidth]{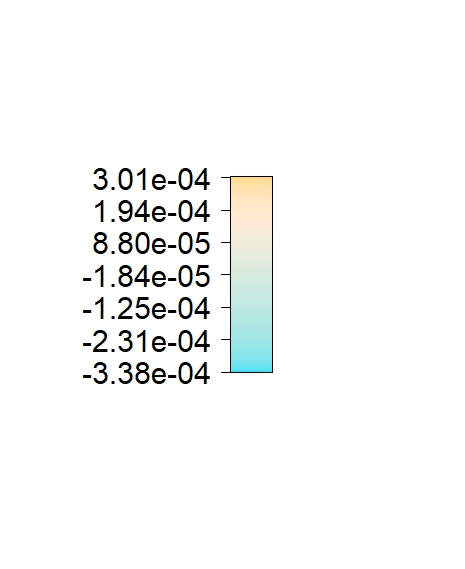}}
\hspace{-1 cm}
\subfigure[The case of $t=0.05,$   $\alpha =1$ and $\beta=0.8$]{\includegraphics[trim = 0.5cm 2cm 1.5cm 3.0cm,clip,width=0.48\linewidth,height=0.41\linewidth]{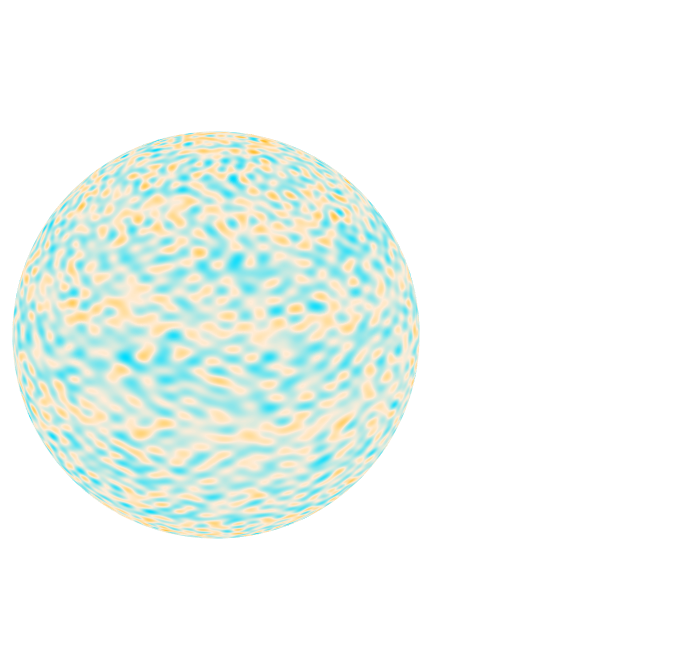}} \hspace{-3 cm}
\subfigure{\includegraphics[trim = 2.0cm 2cm 0.9cm 2.0cm,clip,width=0.2\linewidth,height=0.41\linewidth]{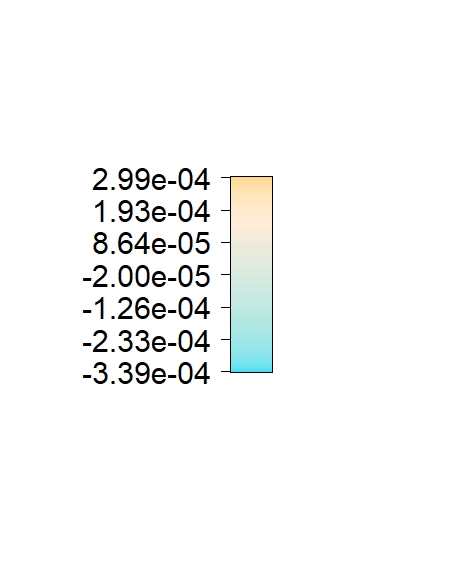}}
    \caption{SFHDRF $T_{H}(x,t)$  for $c=1$ and $D=2.$ }\label{fig2_2}
\end{figure}
\end{example}

\section*{Acknowledgments}

N.~Leonenko and J.~Vaz were supported by  FAPESP grant 22/09201-8 (Brazil).  N.Leonenko's and A.Olenko's research was partially supported under the Australian Research Council's Discovery Projects funding scheme (project number  DP220101680).  N.~Leonenko would like to thank for support  by LMS grant 42997 (UK), the programme “Fractional Differential Equation”s and the programme “Uncertainly Quantification and Modelling of Material”  in Isaac Newton Institute for Mathematical Sciences, Cambridge.

\end{document}